\documentclass[10pt, reqno]{amsart}

\usepackage{geometry}            
\usepackage{amsfonts, amssymb, amsmath, amsthm, amscd}
\usepackage{mathrsfs}
\usepackage{xcolor}
\usepackage[colorlinks=true, linkcolor=red, citecolor=red, urlcolor=blue]{hyperref}
\usepackage{pgf, pgfkeys, tikz}
\usetikzlibrary{arrows}
\usepackage{stackrel}
\usepackage{caption}
\usepackage{units}
\usepackage[T1]{fontenc}
\usepackage[english]{babel}
\usepackage{tabularx}
\usepackage{float}

\emergencystretch=\maxdimen
\theoremstyle{plain}
\newtheorem{theorem}{Theorem}[section]
\newtheorem{definition}[theorem]{Definition}
\newtheorem{proposition}[theorem]{Proposition}
\newtheorem{lemma}[theorem]{Lemma}

\newtheorem*{definition*}{Definition}
\newtheorem*{theorem*}{Theorem}
\newtheorem*{proposition*}{Proposition}
\newtheorem*{lemma*}{Lemma}
\newtheorem*{corollary*}{Corollary}
\newtheorem*{remark*}{Remark}
\newtheorem*{thanks*}{Acknowledgements}

\numberwithin{equation}{section}

\def \bb {\mathbb}

\def \dps {\displaystyle}            
\def \( {\left(}
\def \) {\right)}
\def \[ {\left\lbrack}
\def \] {\right\rbrack}
\def \< {\left\langle}
\def \> {\right\rangle}

\renewcommand{\leq}{\leqslant}            
\renewcommand{\geq}{\geqslant}
\newcommand \eqn[1]{\; #1 \;}

\makeatletter
\def\mathcolor#1#{\@mathcolor{#1}}
\def\@mathcolor#1#2#3{%
    \protect\leavevmode
    \begingroup
        \color#1{#2}#3%
    \endgroup
}
\makeatother

\makeatletter
\def \thebibliography#1 {\section*{References} \list
                                  {[\arabic{enumi}]}{\settowidth \labelwidth{[#1]} \leftmargin \labelwidth
                                   \advance \leftmargin \labelsep
                                   \usecounter{enumi}}
                                   \def \newblock{\hskip .11em plus .33em minus .07em}
                                   \sloppy \clubpenalty4000 \widowpenalty4000
                                   \sfcode`\. = 1000 \relax}

\makeatother

\begin{document}

\title[]{On the elegance of Ramanujan's series for $\dfrac{1}{\pi}$}
\author[]{Chieh-Lei Wong}
\address[]{}
\email[]{\href{mailto:shell_intheghost@hotmail.com}{shell\_intheghost@hotmail.com}.}
\date{\today}

\begin{abstract}
\noindent Re presenting the traditional proof of Srinivasa Ramanujan's own favorite series for the reciprocal of $\pi$ :
\begin{equation*}
\frac{1}{\pi} = \frac{\sqrt{8}}{9801} \sum_{n=0}^{+\infty} \frac{(4n)!}{(n!)^4} \frac{1103 + 26390 n}{396^{4n}} \; \text{,}
\end{equation*}
as well as several other examples of Ramanujan's infinite series. As a matter of fact, the derivation of such formulae has involved specialized knowledge of identities of classical functions and modular functions.
\end{abstract}
\maketitle

\noindent \par The Archimedes' constant $\pi$ appears in many formulae~\cite{BBB1} in various areas of mathematics and physics, such as :
\begin{eqnarray}
\text{James Gregory (1671)} & & \sum_{n=0}^{+\infty} \frac{(-1)^n}{2n+1} = \frac{\pi}{4} \; \text{,} \label{gregory} \\
\text{ Leonhard Euler (1734)} & & \sum_{n=0}^{+\infty} \frac{1}{n^2} = \frac{\pi^2}{6} \; \text{,} \label{euler} \\
\text{Carl Friedrich Gauss (1809)} & &\int_{-\infty}^{+\infty} e^{-x^2} \, dx = \sqrt{\pi} \; \text{,} \label{gauss} \\
\text{Stephen Hawking (1974)} & & T = \frac{1}{8\pi k_B} \frac{\hbar c^3}{GM} \; \text{.}
\end{eqnarray}

\par The irrationality of $\pi$ was first proven by Jean-Henri Lambert in 1761. Finally in 1882, Ferdinand von Lindemann established its transcendence, thus laying to rest the problem of \guillemotleft{} squaring the circle \guillemotright{}.

\section{Aesthetics in mathematics ?}

\noindent \par In 2014, researchers in neurobiology~\cite{ZRBA} from the University College London (in United Kingdom) used functional MRI to image the brain activity of $15$ mathematicians (aged from 22 to 32 years, postgraduate or postdoctoral level, all recruited from colleges in London) when they viewed mathematical formulae. Each subject was given $60$ mathematical formulae - including~(\ref{gregory}), (\ref{euler}) or~(\ref{gauss}) that correspond successively to $\arctan(1)$, $\zeta(2)$ and $\Gamma \( \dfrac{1}{2} \) $ - to study at leisure and rate as ugly [$-1$], neutral [$0$] or beautiful [$+1$]. Note the absence of the nonsimple continued fraction :
\begin{eqnarray}
\text{William Brouncker (1655)} & & \frac{4}{\pi} = 1 + \cfrac{1^2}{2 
+ \cfrac{3^2}{2 
+ \cfrac{5^2}{2
+ \cfrac{7^2}{2
+ \cfrac{9^2}{2 + \ddots}}}}}
\end{eqnarray}
in their list. Results of the study showed that the one most consistently rated as \guillemotleft{} ugly \guillemotright{} was Equation~(14) :
\begin{equation} \label{rs_58}
\frac{1}{\pi} = \frac{\sqrt{8}}{9801} \sum_{n=0}^{+\infty} \frac{(4n)!}{(n!)^4} \frac{1103 + 26390 n}{396^{4n}} \; \text{,}
\end{equation}
an infinite series due to Ramanujan - with an average rating of $\unit{-0,7333}{}$ ! Truly, beauty is in the eye of the beholder.

\par Since the starting point of~(\ref{rs_58}) lays upon the new foundations of elliptic integrals instilled by the works of both Niels Henrik Abel and Carl Gustav J. Jacobi~\cite{Jacobi} in the 19\textsuperscript{th} century, we might remember the premonitory words of Felix Klein :
\begin{quote}
\guillemotleft{} \emph{When I was a student, Abelian functions were, as an effect of the Jacobian tradition, considered the uncontested summit of mathematics, and each of us was ambitious to make progress in this field. And now ? The younger generation hardly knows Abelian functions.} \guillemotright{}
\end{quote}

\par Historically, the identity~(\ref{rs_58}) appeared in~\cite{Ramanujan}. Afterwards, it fell into near oblivion, until the end of 1985 when it was revived in a modern computational context. Seven decades after its publication, Bill Gosper Jr. used it for computing $\unit{17,5.10^6}{}$ decimal digits of $\pi$ - and briefly held the world record. But a significant issue remained : no mathematical proof existed back then that the series~(\ref{rs_58}) actually converges to $\dfrac{1}{\pi}$. It was somehow a leap of faith, yet an educated one. In fact, he verified beforehand that the sum was correct to $10$ million places by comparing this same number of digits of his own calculation to a previous calculation done by Yasumasa Kanada \emph{and al.}

\section{Preliminaries}

\subsection{Jacobi's elliptic integrals}

\noindent \par Let $k \in ]0,1[$ denote the elliptic modulus, then the quantity $k'=\sqrt{1-k^2}$ is called the complementary modulus. Complete elliptic integrals of the first and second kinds are respectively defined as :
\begin{eqnarray}
& & K(k) = \int_0^{\frac{\pi}{2}} \frac{d\theta}{\sqrt{1-k^2 \sin^2 \theta}} = \frac{\pi}{2} \, \phantom{}_2 F_1 \( \begin{array}{c} \dfrac{1}{2}, \dfrac{1}{2} \\ [1.4ex] 1 \end{array} \, \Bigg| \, k^2 \) \\
& \text{and} & E(k) = \int_0^{\frac{\pi}{2}} \sqrt{1-k^2 \sin^2 \theta} \, d\theta = \frac{\pi}{2} \, \phantom{}_2 F_1 \( \begin{array}{c} -\dfrac{1}{2}, \dfrac{1}{2} \\ [1.4ex] 1 \end{array} \, \Bigg| \, k^2 \) \; \text{,}
\end{eqnarray}
while their derivatives are given by :
\begin{equation} \label{derivatives}
\frac{dK}{dk} = \frac{E - k'^2 K}{kk'^2} \quad \text{and} \quad \frac{dE}{dk} = \frac{E - K}{k} \; \text{.}
\end{equation}

\par It is also customary to define the complementary integrals $K'$ and $E'$ as :
\begin{equation*}
K'(k) = K(k') \quad \text{and} \quad E'(k) = E(k') \; \text{.}
\end{equation*}
Finally, these 4 quantities $K$, $K'$, $E$ and $E'$ are linked by the remarkable Legendre relation :
\begin{equation} \label{legendre}
K(k) E'(k) + E(k) K'(k) - K(k) K'(k) = \frac{\pi}{2} \; \text{.}
\end{equation}

\subsection{Jacobi's theta functions}

\noindent \par The theta functions~\cite{Jacobi}, \cite{Lawden} are classically defined as :
\begin{equation} \label{theta_functions}
\theta_2(q) = \sum_{n=-\infty}^{+\infty} q^{ \( n+\frac{1}{2} \) ^2} \quad \text{,} \quad \theta_3(q) = \sum_{n=-\infty}^{+\infty} q^{n^2} \quad \text{and} \quad \theta_4(q) = \sum_{n=-\infty}^{+\infty} (-1)^n q^{n^2} = \theta_3(-q)
\end{equation}
for $|q| < 1$. After rewriting the nome $q$ in terms of the elliptic modulus $k$ :
\begin{equation*}
q = \exp \[ -\pi \frac{K'(k)}{K(k)} \] \; \text{,}
\end{equation*}
it is valuable to regard $k$ as a function of $q$. Thus, we have inversely :
\begin{equation} \label{agm}
k = \frac{\theta_2^2(q)}{\theta_3^2(q)} \quad \text{,} \quad k' = \frac{\theta_4^2(q)}{\theta_3^2(q)} \quad \text{and} \quad K(k) = \frac{\pi}{2} \theta_3^2(q) \; \text{.}
\end{equation}

\subsection{Ramanujan-Weber's class invariants}

\noindent \par Let us introduce Ramanujan's class invariants :
\begin{equation}
G = \( \frac{1}{2 k k'} \) ^{\nicefrac{1}{12}} \quad \text{and} \quad g = \( \frac{k'^2}{2 k} \) ^{\nicefrac{1}{12}} \; \text{,}
\end{equation}
as well as the Klein's absolute invariant :
\begin{equation}
J = \frac{( 4 G^{24} - 1 )^3}{27 G^{24}} = \frac{( 4 g^{24} + 1 )^3}{27 g^{24}} = \frac{4}{27} \frac{\[ 1 - (k k')^2 \] ^3}{(k k')^4} \; \text{.}
\end{equation}

\par In terms of Ramanujan's class invariants, we can explicitly write the elliptic moduli as :
\begin{eqnarray*}
& & k = \frac{1}{2} \( \sqrt{1 + \frac{1}{G^{12}}} - \sqrt{1 - \frac{1}{G^{12}}} \) \quad \text{,} \quad k' = \frac{1}{2} \( \sqrt{1 + \frac{1}{G^{12}}} + \sqrt{1 - \frac{1}{G^{12}}} \) \; \text{,} \\
& \text{or} & k = g^6 \sqrt{g^{12} + \frac{1}{g^{12}}} - g^{12} \quad \text{,} \quad k' = \sqrt{2k} \, g^6 \; \text{.}
\end{eqnarray*}

\subsection{Singular value functions $\lambda^*$ and $\alpha$}

\begin{definition}
Let $\lambda^*(r) = k(e^{-\pi \sqrt{r}})$ be as in~(\ref{agm}), then the singular value function of the second kind is defined by :
\begin{equation}
\alpha(r) = \frac{E'(k)}{K(k)} - \frac{\pi}{4 \big[ K(k) \big]^2}
\end{equation}
for positive $r$.
\end{definition}
Since $\dps{\lim_{r \to +\infty} \lambda^*(r) = 0}$, then $\alpha(r)$ converges to $\dfrac{1}{\pi}$ with exponential rate :
\begin{equation*}
0 < \alpha(r) - \frac{1}{\pi} \leq \sqrt{r} \big[ \lambda^*(r) \big]^2 \leq 16 \sqrt{r} \, e^{-\pi \sqrt{r}} \; \text{.}
\end{equation*}
Using the functional equation~(\ref{legendre}) and the fact that $\dfrac{K' \big( \lambda^*(r) \big)}{K \big( \lambda^*(r) \big)} = \sqrt{r}$, we get :
\begin{equation*}
\alpha(r) = \frac{\pi}{4 \big[ K(k) \big]^2} - \sqrt{r} \[ \frac{E(k)}{K(k)} - 1 \] \; \text{.}
\end{equation*}
On substituting $E$ with the differential equation~(\ref{derivatives}), we may establish that :
\begin{equation*}
\alpha(r) = \frac{1}{\pi} \[ \frac{\pi}{2 K(k)} \] ^2 - \sqrt{r} \[ kk'^2 \frac{1}{K(k)} \frac{dK}{dk} - k^2 \] \; \text{,}
\end{equation*}
so that :
\begin{equation} \label{pi_reciprocal}
\frac{1}{\pi} = \sqrt{r} k k'^2 \[ \( \frac{2}{\pi} \) ^2 K(k) \frac{dK}{dk} \] + \[ \alpha(r) - \sqrt{r} k^2 \] \[ \frac{2}{\pi} K(k) \] ^2
\end{equation}
where $k = \lambda^*(r)$. Also, observe that $\alpha(r)$ is algebraic for $r \in \bb{Q}_+$ (as seen in Tables~\ref{odd_n} and~\ref{even_n} in the next section, or in the computation of $g_{58}^2$ and $k_{58}$ in Subsection~\ref{rs_58_derivation}). Actually, it is well-known that the quantities $\lambda^*(r)$, $G_r$, $g_r$ and $\alpha(r)$ are algebraic numbers expressible by surds when $r$ is a positive rational number.

\subsection{Quadratic and cubic transformations of the hypergeometric function $\protect\phantom{}_2 F_1$}

\noindent \par Let us recall the definition of the hypergeometric series :
\begin{equation}
\phantom{}_2 F_1 \( \begin{array}{c} a,b \\ c \end{array} \, \Big| \, z \) = \sum_{n=0}^{+\infty} \frac{(a)_n (b)_n}{(c)_n} \frac{z^n}{n!} \; \text{,}
\end{equation}
where parameters $a$, $b$ and $c$ are arbitrary complex numbers, and $(a)_n = \dfrac{\Gamma(a+n)}{\Gamma(a)}$ denotes the Pochhammer symbol. However, if and only if the numbers :
\begin{equation} \label{parameters_test}
\pm (1-c) \quad \text{,} \quad \pm (a-b) \quad \text{,} \quad \pm (a+b-c)
\end{equation}
have the property that one of them equals $\dfrac{1}{2}$ or that two of them are equal, then there exists a so-called quadratic transformation. 

\begin{proposition} \label{proposition_1}
For $k \in \[ 0, \dfrac{1}{\sqrt{2}} \] $, we have :
\begin{eqnarray}
& & \frac{2}{\pi} K(k) = \phantom{}_2 F_1 \( \begin{array}{c} \dfrac{1}{4}, \dfrac{1}{4} \\ [1.4ex] 1 \end{array} \, \Bigg| \, (2kk')^2 \) \label{quadratic_g} \\
& \text{and} & \[ \frac{2}{\pi} K(k) \] ^2 = \phantom{}_3 F_2 \( \begin{array}{c} \dfrac{1}{2}, \dfrac{1}{2}, \dfrac{1}{2} \\ [1.4ex] 1,1 \end{array} \, \Bigg| \, (2kk')^2 \) \; \text{.} \label{clausen_g}
\end{eqnarray}
\end{proposition}
\begin{proof}
The first identity~(\ref{quadratic_g}) derives from Kummer's identity :
\begin{equation}
\phantom{}_2 F_1 \( \begin{array}{c} 2a,2b \\ [0.2ex] a+b+\dfrac{1}{2} \end{array} \, \Bigg| \, z \) = \phantom{}_2 F_1 \( \begin{array}{c} a,b \\ [0.2ex] a+b+\dfrac{1}{2} \end{array} \, \Bigg| \, 4z (1-z) \)
\end{equation}
and can be verified by showing that both sides satisfy the appropriate hypergeometric differential equation, are analytic and agree at $0$. The second identity~(\ref{clausen_g}) is a special case of Clausen's product identity :
\begin{equation} \label{clausen_product}
\phantom{}_2 F_1 \( \begin{array}{c} \dfrac{1}{4}+a, \dfrac{1}{4}+b \\ [1.4ex] 1+a+b \end{array} \, \Bigg| \, z \) \, \phantom{}_2 F_1 \( \begin{array}{c} \dfrac{1}{4}-a, \dfrac{1}{4}-b \\ [1.4ex] 1-a-b \end{array} \, \Bigg| \, z \) = \phantom{}_3 F_2 \( \begin{array}{c} \dfrac{1}{2}, \dfrac{1}{2}+a-b, \dfrac{1}{2}-a+b \\ [1.4ex] 1+a+b, 1-a-b \end{array} \, \Bigg| \, z \)
\end{equation}
for hypergeometric functions.
\end{proof}

\par In like fashion, a cubic transformation exists if and only if either two of the numbers in~(\ref{parameters_test}) are equal to $\dfrac{1}{3}$ or if :
\begin{equation*}
1-c = \pm (a-b) = \pm (a+b-c) \; \text{.}
\end{equation*}
Thus, quadratic and cubic transformations of $\phantom{}_2 F_1$ lead to a variety of alternate hypergeometric expressions for $K$ and $K^2$.

\begin{proposition} \label{proposition_2}
We also have :
\begin{align*}
& \frac{2}{\pi} K(k) = \frac{1}{k'} \, \phantom{}_2 F_1 \( \begin{array}{c} \dfrac{1}{4}, \dfrac{1}{4} \\ [1.4ex] 1 \end{array} \, \Bigg| \, -\( \frac{2k}{k'^2} \) ^2 \) & & \hspace{-21.0mm} \text{for } k \in [ 0, \sqrt{2} - 1 ] \; \text{,} \\
& \frac{2}{\pi} K(k) = \frac{1}{\sqrt{k'}} \, \phantom{}_2 F_1 \( \begin{array}{c} \dfrac{1}{4}, \dfrac{1}{4} \\ [1.4ex] 1 \end{array} \, \Bigg| \, -\( \frac{k^2}{2k'} \) ^2 \) & & \hspace{-21.0mm} \text{for } k^2 \in [ 0, 2 ( \sqrt{2} - 1 ) ] \; \text{,} \\
& \frac{2}{\pi} K(k) = \frac{1}{\sqrt{1 + k^2}} \, \phantom{}_2 F_1 \( \begin{array}{c} \dfrac{1}{8}, \dfrac{3}{8} \\ [1.4ex] 1 \end{array} \, \Bigg| \, \( \frac{2}{g^{12} + g^{-12}} \) ^2 \) & & \hspace{-21.0mm} \text{for } k \in [ 0, \sqrt{2} - 1 ] \; \text{,} \\
& \frac{2}{\pi} K(k) = \frac{1}{\sqrt{k'^2 - k^2}} \, \phantom{}_2 F_1 \( \begin{array}{c} \dfrac{1}{8}, \dfrac{3}{8} \\ [1.4ex] 1 \end{array} \, \Bigg| \, -\( \frac{2}{G^{12} - G^{-12}} \) ^2 \) & & \hspace{-21.0mm} \text{for } k \in \[ 0, \frac{1 - \sqrt{\sqrt{2} - 1}}{2^{\nicefrac{3}{4}}} \] \; \text{,} \\
\text{and} \quad & \frac{2}{\pi} K(k) = \frac{1}{\big[ 1 - (kk')^2 \big] ^{\nicefrac{1}{4}}} \, \phantom{}_2 F_1 \( \begin{array}{c} \dfrac{1}{12}, \dfrac{5}{12} \\ [1.4ex] 1 \end{array} \, \Bigg| \, \frac{1}{J} \) & & \hspace{-21.0mm} \text{for } k \in \[ 0, \frac{1}{\sqrt{2}} \] \; \text{.}
\end{align*}
\end{proposition}
\begin{proof}
See e.g.~\cite{Goursat} or~\cite{Erdelyi}.
\end{proof}

\begin{proposition} \label{proposition_3}
For $k$ restricted as in Proposition~\ref{proposition_2} :
\begin{eqnarray*}
& & \[ \frac{2}{\pi} K(k) \] ^2 = \frac{1}{k'^2} \, \phantom{}_3 F_2 \( \begin{array}{c} \dfrac{1}{2}, \dfrac{1}{2}, \dfrac{1}{2} \\ [1.4ex] 1,1 \end{array} \, \Bigg| \, -\( \frac{2k}{k'^2} \) ^2 \) \; \text{,} \\
& & \[ \frac{2}{\pi} K(k) \] ^2 = \frac{1}{k'} \, \phantom{}_3 F_2 \( \begin{array}{c} \dfrac{1}{2}, \dfrac{1}{2}, \dfrac{1}{2} \\ [1.4ex] 1,1 \end{array} \, \Bigg| \, -\( \frac{k^2}{2k'} \) ^2 \) \; \text{,} \\
& & \[ \frac{2}{\pi} K(k) \] ^2 = \frac{1}{1 + k^2} \, \phantom{}_3 F_2 \( \begin{array}{c} \dfrac{1}{4}, \dfrac{3}{4}, \dfrac{1}{2} \\ [1.4ex] 1,1 \end{array} \, \Bigg| \, \( \frac{2}{g^{12} + g^{-12}} \) ^2 \) \; \text{,} \\
& & \[ \frac{2}{\pi} K(k) \] ^2 = \frac{1}{k'^2 - k^2} \, \phantom{}_3 F_2 \( \begin{array}{c} \dfrac{1}{4}, \dfrac{3}{4}, \dfrac{1}{2} \\ [1.4ex] 1,1 \end{array} \, \Bigg| \, -\( \frac{2}{G^{12} - G^{-12}} \) ^2 \) \; \text{,} \\
& \text{and} & \[ \frac{2}{\pi} K(k) \] ^2 = \frac{1}{\sqrt{1 - (kk')^2}} \, \phantom{}_3 F_2 \( \begin{array}{c} \dfrac{1}{6}, \dfrac{5}{6}, \dfrac{1}{2} \\ [1.4ex] 1,1 \end{array} \, \Bigg| \, \frac{1}{J} \) \; \text{.}
\end{eqnarray*}
\end{proposition}
\begin{proof}
Apply the Clausen's identity~(\ref{clausen_product}) to Proposition~\ref{proposition_2}.
\end{proof}

\par In each case, we have provided series for $\dfrac{2}{\pi} K$ and $\( \dfrac{2}{\pi} K \) ^2$ in terms of the Ramanujan's invariants. Indeed, we have :
\begin{equation*}
\[ \frac{2}{\pi} K(k) \] ^2 = m(k) F \big( \varphi(k) \big)
\end{equation*}
for algebraic $m$ and $\varphi$, while $F (\varphi)$ has a hypergeometric-type power series expansion $\dps{\sum_{n=0}^{+\infty} a_n \varphi^n}$. Then :
\begin{equation*}
\( \frac{2}{\pi} \) ^2 K \frac{dK}{dk} = \frac{1}{2} \[ \frac{dm}{dk} F + m \frac{d\varphi}{dk} \frac{dF}{d\varphi} \]
\end{equation*}
and substitution in~(\ref{pi_reciprocal}) lead to :
\begin{equation} \label{general_form}
\frac{1}{\pi} = \sum_{n=0}^{+\infty} a_n \left\{ \frac{1}{2} \sqrt{r} k k'^2 \frac{dm}{dk} + \[ \alpha(r) - \sqrt{r} k^2 \] m + \frac{1}{2} n \sqrt{r} k k'^2 \frac{m}{\varphi} \frac{d\varphi}{dk} \right\} \varphi^n \; \text{.}
\end{equation}
Thus for rational $r$, the braced term in~(\ref{general_form}) is of the form $A+nB$ with $A$ and $B$ algebraic.

\section{Examples of hypergeometric-like series representations for $\dfrac{1}{\pi}$}

\subsection{Deriving Ramanujan's series for $\dfrac{1}{\pi}$}

\noindent \par By combining Propositions~\ref{proposition_1}, \ref{proposition_2} and~\ref{proposition_3} with the formula~(\ref{general_form}), it is now straightforward to build the next 6 series :
\begin{eqnarray}
(\text{series in } G_N) & & \frac{1}{\pi} = \sum_{n=0}^{+\infty} \Bigg[ \frac{1}{n!} \( \frac{1}{2} \) _n \Bigg]^3 \[ \alpha(N) - \sqrt{N} k_N^2 + n \sqrt{N} \( k_N'^2 - k_N^2 \) \] \( \frac{1}{G_N^{12}} \) ^{2n} \label{upper_gn} \\
(\text{series in } g_N) & & \frac{1}{\pi} = \sum_{n=0}^{+\infty} (-1)^n \Bigg[ \frac{1}{n!} \( \frac{1}{2} \) _n \Bigg]^3 \[ \frac{\alpha(N)}{k_N'^2} + n \sqrt{N} \frac{1 + k_N^2}{k_N'^2} \] \( \frac{1}{g_N^{12}} \) ^{2n} \\
(\text{series in } g_{4N} = 2^{\nicefrac{1}{4}} g_N G_N) & & \frac{1}{\pi} = \sum_{n=0}^{+\infty} (-1)^n \Bigg[ \frac{1}{n!} \( \frac{1}{2} \) _n \Bigg]^3 \left\{ \[ \alpha(N) - \sqrt{N} \frac{k_N^2}{2} \] \frac{1}{k_N'} + n \sqrt{N} \( k_N' + \frac{1}{k_N'} \) \right\} \( \frac{1}{g_{4N}^{12}} \) ^{2n}
\end{eqnarray}
On setting $x_N = \dfrac{2}{g_N^{12} + g_N^{-12}} = \dfrac{4 k_N k_N'^2}{(1 + k_N^2)^2}$ and $y_N = \dfrac{2}{G_N^{12} - G_N^{-12}} = \dfrac{4 k_N k_N'}{1 - (2 k_N k_N')^2}$ :
\begin{eqnarray}
(\text{series in } x_N) & & \frac{1}{\pi} = \sum_{n=0}^{+\infty} \frac{\( \dfrac{1}{4} \) _n \( \dfrac{1}{2} \) _n \( \dfrac{3}{4} \) _n}{(n!)^3} \[ \frac{\alpha(N)}{x_N ( 1 + k_N^2 )} - \frac{\sqrt{N}}{4 g_N^{12}} + n \sqrt{N} \frac{g_N^{12} - g_N^{-12}}{2} \] x_N^{2n+1} \label{xn} \\
(\text{series in } y_N) & & \frac{1}{\pi} = \sum_{n=0}^{+\infty} (-1)^n \frac{\( \dfrac{1}{4} \) _n \( \dfrac{1}{2} \) _n \( \dfrac{3}{4} \) _n}{(n!)^3} \[ \frac{\alpha(N)}{y_N ( k_N'^2 - k_N^2 )} + \sqrt{N} \frac{k_N^2 G_N^{12}}{2} + n \sqrt{N} \frac{G_N^{12} + G_N^{-12}}{2} \] y_N^{2n+1} \label{yn}
\end{eqnarray}
And eventually the series in $J_N$ :
\begin{equation} \label{jn}
\frac{1}{\pi} = \frac{1}{3 \sqrt{3}} \sum_{n=0}^{+\infty} \frac{\( \dfrac{1}{6} \) _n \( \dfrac{1}{2} \) _n \( \dfrac{5}{6} \) _n}{(n!)^3} \left\{ 2 \[ \alpha(N) - \sqrt{N} k_N^2 \] \( 4 G_N^{24} - 1 \) + \sqrt{N} \sqrt{1 - \frac{1}{G_N^{24}}} + 2n \sqrt{N} \( 8 G_N^{24} + 1 \) \sqrt{1 - \frac{1}{G_N^{24}}} \right\} \( \frac{1}{J_N^{\nicefrac{1}{2}}} \) ^{2n+1}
\end{equation}
that is valid for $N>1$.

\subsection{Applications}

\noindent \par Let us first evaluate the Pochhammer symbols. It is well-known that :
\begin{equation*}
\frac{1}{n!} \( \frac{1}{2}\) _n = \frac{1}{4^n} \binom{2n}{n}
\end{equation*}
in terms of the central binomial coefficient. For the remaining symbols, we may require the following lemma :

\begin{lemma}
For any $n \in \bb{N}$, we have :
\begin{eqnarray*}
& & \( \frac{1}{4} \) _n \( \frac{1}{2} \) _n \( \frac{3}{4} \) _n = \frac{1}{4^{4n}} \frac{(4n)!}{n!} \; \text{,} \\
& \text{as well as} & \( \frac{1}{6} \) _n \( \frac{1}{2} \) _n \( \frac{5}{6} \) _n = \frac{1}{12^{3n}} \frac{(6n)!}{(3n)!} \; \text{.}
\end{eqnarray*}
\end{lemma}
\begin{proof}
Let $p,q \in \bb{N}^*$, observe that :
\begin{equation*}
\( \frac{p}{q} \) _n = \frac{1}{q^n} \prod_{m=1}^n \big[ p + (m-1) q \big] \; \text{.}
\end{equation*}
Subsequently :
\begin{eqnarray*}
& & \( \frac{1}{4} \) _n \( \frac{1}{2} \) _n \( \frac{3}{4} \) _n = \frac{1}{4^{3n}} \prod_{m=1}^n (4m-3) (4m-2) (4m-1) = \frac{1}{4^{4n}} \frac{(4n)!}{n!} \; \text{,} \\
& \text{whereas} & \( \frac{1}{6} \) _n \( \frac{1}{2} \) _n \( \frac{5}{6} \) _n = \frac{1}{6^{3n}} \prod_{m=1}^n (6m-5) (6m-3) (6m-1) = \frac{1}{12^{3n}} \frac{(6n)!}{(3n)!} \; \text{.}
\end{eqnarray*}
\end{proof}

\begin{definition}
Let $d$ be a square-free integer, we consider the real quadratic number field $\Bbbk = \bb{Q}(\sqrt{d})$. If $\Delta_{\Bbbk}$ denotes the discriminant of $\Bbbk$ i.e. :
\begin{equation*}
\Delta_{\Bbbk} = \left\{ \begin{array}{cl} d & \text{if } d = 1 \pmod{4} \\ 4d & \text{if } d = 2,3 \pmod{4} \end{array} \right. \; \text{,}
\end{equation*}
then the fundamental unit $u_d > 1$ is uniquely characterized as the minimal real number :
\begin{equation}
u_d = \frac{a+b \sqrt{\Delta_{\Bbbk}}}{2}
\end{equation}
where $(a,b)$ is the smallest solution to $m^2 - \Delta_{\Bbbk} n^2 = \pm 4$ in positive integers. This equation is essentially Pell-Fermat's equation.
\end{definition}

\par Of course, the most challenging part in the formula~(\ref{general_form}) lies in the evaluation of the singular value function $\alpha$. For positive rational $r$, many values of $\alpha(r)$ are obtainable. But details would be slightly beyond the scope of this paper, with deep roots in number-theoretic objects and techniques such as modular equations, multipliers, modular forms, the Dedekind's $\eta$ function, and so on. Alternatively, we shall rely on Weber~\cite{Weber} and Ramanujan~\cite{Ramanujan}. Some of the nicest singular values are collected in the following tables.

\begin{table}[H]
{\renewcommand{\arraystretch}{2.4}
\centering
    \begin{tabular}{|>{\centering\arraybackslash}p{6.0mm}||c|c|c|c|}
        \hline
        $N$ & $k_N$ & $\dfrac{1}{G_N^{12}}$ & $\alpha(N)$ & $u_N$ \\ [0.8ex]
        \hline \hline
        $3$ & $\dfrac{\sqrt{3} - 1}{2 \sqrt{2}}$ & $\dfrac{1}{2}$ & $\dfrac{\sqrt{3} - 1}{2}$ & $2 + \sqrt{3}$ \\ [0.8ex]
        \hline
        $5$ & $\dfrac{\sqrt{\sqrt{5} - 1} - \sqrt{3 - \sqrt{5}}}{2}$ & $\bigg( \dfrac{\sqrt{5} - 1}{2} \bigg)^3$ & $\dfrac{\sqrt{5} - \sqrt{2 \sqrt{5} - 2}}{2}$ & $\dfrac{1 + \sqrt{5}}{2}$ \\ [0.8ex]
        \hline
        $7$ & $\dfrac{3 - \sqrt{7}}{4 \sqrt{2}}$ & $\dfrac{1}{8}$ & $\dfrac{\sqrt{7} - 2}{2}$ & $8 + 3 \sqrt{7}$ \\ [0.8ex]
        \hline
        $9$ & $\dfrac{( \sqrt{2} - 3^{\nicefrac{1}{4}} ) ( \sqrt{3} - 1 )}{2}$ & $( 2 - \sqrt{3} )^2$ & $\dfrac{3 - 3^{\nicefrac{3}{4}} \sqrt{2} ( \sqrt{3} - 1 )}{2}$ & $-$ \\ [0.4ex]
        \hline
        $13$ & $\dfrac{\sqrt{10 \sqrt{13} - 34} - 5 + \sqrt{13}}{2 \sqrt{2}}$ & $\bigg( \dfrac{\sqrt{13} - 3}{2} \bigg)^3$ & $\dfrac{\sqrt{13} - \sqrt{74 \sqrt{13} - 258}}{2}$ & $\dfrac{3 + \sqrt{13}}{2}$ \\ [0.8ex]
        \hline
        $15$ & $\dfrac{( 2 - \sqrt{3} ) ( 3 - \sqrt{5} ) ( \sqrt{5} - \sqrt{3} )}{8 \sqrt{2}}$ & $\dfrac{1}{8} \bigg( \dfrac{\sqrt{5} - 1}{2} \bigg) ^4$ & $\dfrac{\sqrt{15} - \sqrt{5} - 1}{2}$ & $4 + \sqrt{15}$ \\ [0.8ex]
        \hline
        $25$ & $\dfrac{( \sqrt{5} - 2 ) ( 3 - 2 \times 5^{\nicefrac{1}{4}} )}{\sqrt{2}}$ & $\bigg( \dfrac{\sqrt{5} - 1}{2} \bigg)^{12}$ & $\dfrac{5 \[ 1 - 2 \times 5^{\nicefrac{1}{4}} ( 7 - 3 \sqrt{5} ) \] }{2}$ & $-$ \\ [0.8ex]
        \hline
        $37$ & $\dfrac{\sqrt{290 \sqrt{37} - 1762} + 29 - 5 \sqrt{37}}{2 \sqrt{2}}$ & $( \sqrt{37} - 6 )^3$ & $\dfrac{\sqrt{37} - ( 171 - 25 \sqrt{37} ) \sqrt{\sqrt{37} - 6}}{2}$ & $6 + \sqrt{37}$ \\ [0.8ex]
        \hline
    \end{tabular}
    \bigskip
    \caption{Selected singular values, class invariants $G_N$ and fundamental units $u_N$ for $N$ odd.}
    \label{odd_n}}
\end{table}

\par In Table~\ref{odd_n}, observe that $G_N^4 = u_N$ for $N=5$, $13$ and $37$.

\begin{table}[H]
{\renewcommand{\arraystretch}{2.4}
\centering
    \begin{tabular}{|>{\centering\arraybackslash}p{6.0mm}||c|c|c|c|c|}
        \hline
        $N$ & $k_N$ & $\dfrac{1}{g_N^{12}}$ & $\alpha(N)$ & $u_{\nicefrac{N}{2}}$ & $u_N$ \\ [0.8ex]
        \hline \hline
        $2$ & $\sqrt{2} - 1$ & $1$ & $\sqrt{2} - 1$ & $-$ & $1 + \sqrt{2}$ \\
        \hline
        $6$ & $( 2 - \sqrt{3} ) ( 5 - 2 \sqrt{6} )^{\nicefrac{1}{2}}$ & $( \sqrt{2} - 1 )^2$ & $( \sqrt{2} + 1 ) ( 2 - \sqrt{3} ) ( 5 - 2 \sqrt{6} )^{\nicefrac{1}{2}} ( 3 - \sqrt{2} )$ & $2 + \sqrt{3}$ & $5 + 2 \sqrt{6}$ \\
        \hline
        $10$ & $( \sqrt{2} - 1 )^2 ( \sqrt{10} - 3 )$ & $\bigg( \dfrac{\sqrt{5} - 1}{2} \bigg)^6$ & $\bigg( \dfrac{\sqrt{5} + 1}{2} \bigg)^3 ( \sqrt{2} - 1 )^2 ( \sqrt{10} - 3 ) ( 3 \sqrt{5} - 4 )$ & $\dfrac{1 + \sqrt{5}}{2}$ & $3 + \sqrt{10}$ \\ [0.8ex]
        \hline
        $18$ & $( 7 - 4 \sqrt{3} ) ( 5 \sqrt{2} - 7 )$ & $( \sqrt{3} - \sqrt{2} )^4$ & $3 ( \sqrt{3} + \sqrt{2} )^2 ( 7 - 4 \sqrt{3} ) ( 5 \sqrt{2} - 7 ) ( 7 - 2 \sqrt{6} )$ & $-$ & $1 + \sqrt{2}$ \\
        \hline
        $22$ & $( 10 - 3 \sqrt{11} ) ( 197 - 42 \sqrt{22} )^{\nicefrac{1}{2}}$ & $( \sqrt{2} - 1 )^6$ & $( \sqrt{2} + 1 )^3 ( 10 - 3 \sqrt{11} ) ( 197 - 42 \sqrt{22} )^{\nicefrac{1}{2}} ( 33 - 17 \sqrt{2} )$ & $10 + 3 \sqrt{11}$ & $197 + 42 \sqrt{22}$ \\
        \hline
        $58$ & $( \sqrt{2} - 1 )^6 ( 13 \sqrt{58} - 99 )$ & $\bigg( \dfrac{\sqrt{29} - 5}{2} \bigg) ^6$ & $3 \bigg( \dfrac{\sqrt{29} + 5}{2} \bigg) ^3 ( \sqrt{2} - 1 )^6 ( 13 \sqrt{58} - 99 ) ( 33 \sqrt{29} - 148 )$ & $\dfrac{5 + \sqrt{29}}{2}$ & $99 + 13 \sqrt{58}$ \\ [0.8ex]
        \hline
    \end{tabular}
    \bigskip
    \caption{Selected singular values, class invariants $g_N$ and fundamental units $u_{\nicefrac{N}{2}}$ and $u_N$ for $N$ even.}
    \label{even_n}}
\end{table}

\par For $N=6$, $10$, $18$, $22$ and $58$, observe that the values of the function $\alpha$ in Table~\ref{even_n} are all expressed in the form $\alpha(N) = g_N^6 k_N f_N$, where $f_N$ is an element of some quadratic field $\bb{Q}(\sqrt{d})$ with $d \mid N$.

\par Many more singular moduli are given in~\cite{BBB3} or~\cite{Petrovic}.

\subsubsection{The case $N=7$}

\noindent \par Table~\ref{odd_n} provides :
\begin{equation*}
G_7^{12} = 8 \quad \text{and} \quad \alpha(7) = \frac{\sqrt{7}}{2} - 1 \; \text{,}
\end{equation*}
so that $k_7^2 = \dfrac{8 - 3 \sqrt{7}}{16}$ . By putting these values in the series~(\ref{upper_gn}) which is valid for $N>1$, we obtain :
\begin{equation}
\frac{1}{\pi} = \frac{1}{16} \sum_{n=0}^{+\infty} \frac{\big[ (2n)! \big]^3}{(n!)^6} \frac{5 + 7 \times 6 n}{64^{2n}} \; \text{.}
\end{equation}
This is equivalent to Equation~(29) in Ramanujan's original paper~\cite{Ramanujan}. Being composed of fractions whose numerators grow like $\sim 2^{6n}$ and whose denominators are exactly $16 \times 2^{12n}$, the above series can be employed to calculate the second block of $n$ binary digits of $\pi$ without calculating the first $n$ binary digits.

\par Note that the series~(\ref{yn}) is valid for $N \geq 4$. On using the invariant $y_7 = \dfrac{16}{63}$ in~(\ref{yn}), we get :
\begin{equation}
\frac{1}{\pi} = \frac{1}{9 \sqrt{7}} \sum_{n=0}^{+\infty} (-1)^n \frac{(4n)!}{(n!)^4} \frac{8 + 65 n}{63^{2n}} \; \text{,}
\end{equation}
while combining $J_7 = \( \dfrac{85}{4} \) ^3$ with~(\ref{jn}) shall produce the series :
\begin{equation}
\frac{1}{\pi} = \frac{18}{85} \sqrt{\frac{3}{85}} \sum_{n=0}^{+\infty} \frac{(6n)!}{(3n)! (n!)^3} \frac{8 + 7 \times 19 n}{255^{3n}} \; \text{.}
\end{equation}
One may recognize Equation~(34) of~\cite{Ramanujan} which adds $4$ decimal digits a term.

\subsubsection{The case $N=37$}

\noindent \par Let us recall that $G_{37}^4 = u_{37} = 6 + \sqrt{37}$. From Table~\ref{odd_n}, we get :
\begin{equation*}
y_{37} = \frac{2}{G_{37}^{12} - G_{37}^{-12}} = \frac{1}{882} \quad \text{,} \quad \frac{G_{37}^{12} + G_{37}^{-12}}{2} = 145 \sqrt{37} \quad \text{,} \quad \alpha(37) = \frac{\sqrt{37} - ( 171 - 25 \sqrt{37} ) G_{37}^{-2}}{2} \; \text{,}
\end{equation*}
as well as :
\begin{equation*}
k_{37}^2 = \frac{1}{2} \( 1 - \frac{1}{G_{37}^6} \sqrt{G_{37}^{12} - \frac{1}{G_{37}^{12}}} \) = \frac{1}{2} \( 1 - \frac{42}{G_{37}^6} \) \quad \Longrightarrow \quad \frac{k_{37}^2 G_{37}^{12}}{2} = \frac{G_{37}^6 ( G_{37}^6 - 42 )}{4} \; \text{.}
\end{equation*}
Consequently :
\begin{eqnarray*}
\frac{\alpha(37)}{y_{37} ( k_{37}'^2 - k_{37}^2 )} + \sqrt{37} \, \frac{k_{37}^2 G_{37}^{12}}{2} & = & \frac{21}{2} \[ \sqrt{37} - ( 171 - 25 \sqrt{37} ) G_{37}^{-2} \] G_{37}^6 + \sqrt{37} \, \frac{( G_{37}^6 - 42 ) G_{37}^6}{4} \\
& = & \frac{G_{37}^4}{4} \[ -42 ( 171 - 25 \sqrt{37} ) + \sqrt{37} \, G_{37}^8 \] \\
& = & \frac{6 + \sqrt{37}}{4} \[ -42 ( 171 - 25 \sqrt{37} ) + \sqrt{37} ( 6 + \sqrt{37} )^2 \] \eqn{=} \frac{1123}{4} \; \text{.}
\end{eqnarray*}
Putting these numerical values into~(\ref{yn}) yields :
\begin{eqnarray}
\frac{1}{\pi} & = & \sum_{n=0}^{+\infty} \frac{(-1)^n}{4^{4n}} \frac{(4n)!}{(n!)^4} \[ \frac{\alpha(37)}{y_{37} ( k_{37}'^2 - k_{37}^2 )} + \sqrt{37} \, \frac{k_{37}^2 G_{37}^{12}}{2} + n \sqrt{37} \, \frac{G_{37}^{12} + G_{37}^{-12}}{2} \] y_{37}^{2n+1} \nonumber \\
& = & \frac{1}{3528} \sum_{n=0}^{+\infty} (-1)^n \frac{(4n)!}{(n!)^4} \frac{1123 + 37 \times 580 n}{14112^{2n}}
\end{eqnarray}
which can be identified with Equation~(39) of~\cite{Ramanujan}.

\subsubsection{The case $N=58$} \label{rs_58_derivation}

\noindent \par Let $r \in \bb{Q}_+^*$, it turns out that :
\begin{equation}
\sideset{}{'} \sum_{m,n=-\infty}^{+\infty} \frac{(-1)^m}{m^2 + r n^2} = -\frac{\pi}{\sqrt{r}} \log (2 g_r^4) \; \text{.}
\end{equation}
Since this zeta sum over a $2$-dimensional lattice (with the exception of the origin) can be decomposed into a sum of products of $L$-series, we have :
\begin{equation*}
\sideset{}{'} \sum_{m,n=-\infty}^{+\infty} \frac{(-1)^{m+1}}{m^2 + 58 n^2} = \frac{\pi}{\sqrt{58}} \log 2 + \sum_{d \mid 29} \[ 1 - \( \frac{2}{d} \) \] L_{-\frac{232}{d}}(1) L_d(1) = \frac{\pi}{\sqrt{58}} \log 2 + 2 L_{-8}(1) L_{29}(1)
\end{equation*}
where $\( \dfrac{2}{d} \) $ denotes the Kronecker symbol of $2$ and $d>0$. Hence :
\begin{equation*}
\frac{\pi}{\sqrt{58}} \log (2 g_{58}^4) = \frac{\pi}{\sqrt{58}} (\log 2 + 2 \log u_{29}) \quad \Longrightarrow \quad g_{58}^2 = u_{29} = \frac{5 + \sqrt{29}}{2} \; \text{.}
\end{equation*}

\par From the relation :
\begin{equation}
\sideset{}{'} \sum_{m,n=-\infty}^{+\infty} \frac{(-1)^m}{m^2 + 2r n^2} - 4 \sideset{}{'} \sum_{m,n=-\infty}^{+\infty} \frac{(-1)^m}{m^2 + 8r n^2} = -\frac{\pi}{\sqrt{2r}} \log \( \frac{k_r}{4} \) \; \text{,}
\end{equation}
we may similarly deduce that $k_{58} = \dfrac{1}{u_2^6 u_{58}} = ( \sqrt{2} - 1 )^6 ( 13 \sqrt{58} - 99 )$. So $k_{58} + \dfrac{1}{k_{58}} = 198 \sqrt{2} ( 13 \sqrt{29} + 70 )$.

\par On inserting now the numerical values :
\begin{equation*}
x_{58} = \frac{2}{g_{58}^{12} + g_{58}^{-12}} = \frac{1}{9801} \quad \text{,} \quad \frac{g_{58}^{12} - g_{58}^{-12}}{2} = 1820 \sqrt{29} \quad \text{,} \quad \alpha(58) = 3 g_{58}^6 k_{58} (33 \sqrt{29} - 148) \; \text{,}
\end{equation*}
and :
\begin{eqnarray*}
\frac{\alpha(58)}{x_{58} ( 1 + k_{58}^2 )} - \frac{\sqrt{58}}{4 g_{58}^{12}} & = & \frac{3 (33 \sqrt{29} - 148)}{x_{58} (k_{58} + k_{58}^{-1})} \( \frac{\sqrt{29} + 5}{2} \) ^3 - \frac{1}{2} \sqrt{\frac{29}{2}} \( \frac{\sqrt{29} - 5}{2} \) ^6 \\
& = & \frac{1}{2 \sqrt{2}} \[ 297 (33 \sqrt{29} - 148) - \sqrt{29} (9801 - 1820 \sqrt{29}) \] \eqn{=} 2 \sqrt{2} \times 1103
\end{eqnarray*}
into the series~(\ref{xn}), we find that :
\begin{eqnarray}
\frac{1}{\pi} & = & \sum_{n=0}^{+\infty} \frac{1}{4^{4n}} \frac{(4n)!}{(n!)^4} \[ \frac{\alpha(58)}{x_{58} ( 1 + k_{58}^2 )} - \frac{\sqrt{58}}{4 g_{58}^{12}} + n \sqrt{58} \, \frac{g_{58}^{12} - g_{58}^{-12}}{2} \] x_{58}^{2n+1} \nonumber \\
& = & \frac{2 \sqrt{2}}{9801} \sum_{n=0}^{+\infty} \frac{(4n)!}{(n!)^4} \frac{1103 + 29 \times 910 n}{396^{4n}} \; \text{.}
\end{eqnarray}
This concludes the proof of Equation~(44) in~\cite{Ramanujan}. As observed by Ramanujan himself, the series~(\ref{rs_58}) is extremely rapidly convergent by adding $8$ decimal digits a term !

\par As an exercise, the reader is encouraged to determine the other series of~\cite{Ramanujan} with the singular values in Tables~\ref{odd_n} and~\ref{even_n}. A solution is provided in the companion file \url{https://clwmypage.files.wordpress.com/2021/01/ramanujan-reciprocal-pi.pdf}.

\section{Conclusion}

\noindent \par Srinivasa Ramanujan recorded the bulk of his mathematical results in several notebooks of looseleaf paper and mostly written up without proofs. Hence, his works were often shrouded in a veil of divine magic and mystery. As being a deeply religious Hindu, he credited his substantial capacities to divinity, and stated that formulae were revealed to him by his family goddess, Namagiri Thayar. During the 20\textsuperscript{th} century, the many results in \emph{Ramanujan's Notebooks} inspired numerous papers by later mathematicians trying to prove what he had previously found.

\par As demonstrated, the general formula~(\ref{general_form}) produces multiple reciprocal series for $\pi$ in terms of the function $\alpha(r)$ and related modular quantities. Thus, we showed that the amazing sum~(\ref{rs_58}) is a specialization (when $N=58$) of~(\ref{general_form}) coupled with the invariant $\varphi(k) = \[ \dfrac{4k (1-k^2)}{(1+k^2)^2} \] ^2$.

\par For the sake of simplicity, we have intentionally skipped here some technical aspects, namely about modular equations of order~$p$ (with $p$ prime), modular forms, Eisenstein series, the Dedekind's $\eta$ function, etc. References~\cite{BBB2} and~\cite{BBC} (as well as multiple references therein) are accessible expository papers in connection with Ramanujan's series for $\dfrac{1}{\pi}$. For a deeper insight, material based on the context of elliptic and modular curves can be found e.g. in~\cite{BBB3}, \cite{CC} or~\cite{CG}.

\par This leads naturally to an other famous instance of Ramanujan-Sato series, to wit :
\begin{eqnarray} \label{chudnovsky}
\text{Chudnovsky (1988)} & & \frac{1}{\pi} = 12 \sum_{n=0}^{+\infty} (-1)^n \frac{(6n)!}{(3n)! (n!)^3} \frac{13591409 + 163 \times 3344418 n}{640320^{3 \( n+\frac{1}{2} \) }}
\end{eqnarray}
when $N=163$. On the quest for digits of $\pi$, the series~(\ref{chudnovsky}) was used by Alexander J. Yee and Shigeru Kondo to calculate more than $\unit{12,1.10^{12}}{}$ decimal places for a new record-breaking computation in 2013.

\par It was only recently that Heng Huat Chan and Shaun Cooper~\cite{CC} discovered a general approach that used the underlying modular congruence subgroup $\Gamma_0(N)$ to generate a set of all-new Ramanujan-Sato series, such as :
\begin{eqnarray*}
\text{Chan \& Cooper (2012)} & & \frac{1}{\pi} = 2 \sqrt{2} \sum_{n=0}^{+\infty} \[ \sum_{m=0}^n \frac{(-1)^{n-m}}{64^m} \frac{(4m)!}{(m!)^4} \binom{n+m}{n-m} \] \[ -24184 + 9801 \sqrt{29} \( n + \frac{1}{2} \) \] \( \frac{\sqrt{29} - 5}{2} \) ^{12 \( n + \frac{1}{2} \) }
\end{eqnarray*}
which can be considered as a counterpart of~(\ref{rs_58}).

\bigskip

\begin{thanks*}
The author would like to thank the anonymous referee for his / her constructive comments and suggestions which significantly improved the present manuscript.
\end{thanks*}

\end{document}